\documentclass{article}[12pt]
\usepackage{amsfonts,amsmath,amsthm,yfonts,mathdots,
latexsym,graphicx,amssymb,hyperref,url,fancyhdr,lineno,tikz,float}
\usepackage{amsfonts,amsmath,latexsym,hyperref,graphicx,amssymb,amsthm,url,lineno,float}
\usepackage[abbrev,msc-links]{amsrefs}  
\usetikzlibrary{arrows}
\usepackage{rotating}
\usetikzlibrary{calc}
\usetikzlibrary{positioning}
\usetikzlibrary{patterns}
\usetikzlibrary{shapes}
\usepackage{subcaption}
\usepackage{gensymb}

\title{A Euclidean Ramsey result in the plane}

\author{Sergei Tsaturian\thanks{University of Manitoba, \url{tsaturis@myumanitoba.ca}}}
\date{4 April 2017}

\begin{document}
\maketitle

\begin{abstract}
An old question in Euclidean Ramsey theory asks, if the points in the plane are red-blue coloured, does there always exist a red pair of points at unit distance or five blue points in line separated by unit distances?
An elementary proof answers this question in the affirmative.
\end{abstract}

\newtheorem{theorem}{Theorem}[section]
\newtheorem{lemma}[theorem]{Lemma}
\newtheorem{corollary}[theorem]{Corollary}
\newtheorem{conjecture}{Conjecture}
\newtheorem{definition}[theorem]{Definition}
\newtheorem{problem} [theorem]{Problem}
\newtheorem{proposition} [theorem]{Proposition}
\newtheorem{example}[theorem]{Example}
\newtheorem{question}{Question}


\section{Introduction}
Many problems in Euclidean Ramsey theory ask, for some $d\in\mathbb{Z}^+$, if $E^d$ is coloured with $r\geq 2$ colours, does there exist a colour class containing some desired geometric structure? Research in Euclidean Ramsey theory was surveyed in \cites{erdos1, erdos2, erdos3} by Erd\H os, Graham, Montgomery, Rothschild, Spencer, and Straus; for a more recent survey, see Graham \cite{graham}.\\
Say that two geometric configurations are congruent iff there exists an isometry (distance preserving bijection) between them. For $d\in\mathbb{Z}^+$, and geometric configurations $F_1$, $F_2$, let the notation $\mathbb{E}^d\rightarrow (F_1,F_2)$ mean that for any red-blue coloring of $\mathbb{E}^d$, either the red points contain a congruent copy of $F_1$, or the blue points contain a congruent copy of $F_2$.
For a positive integer $i$, denote by $\ell_i$ the configuration of $i$ collinear points with distance $1$ between consecutive points. One of the results in \cite{erdos2} states that \begin{equation}\label{e:l2l4}\mathbb{E}^2\rightarrow (\ell_2,\ell_4).\end{equation}
In the same paper, it was asked if $\mathbb{E}^2\rightarrow (\ell_2,\ell_5)$, or perhaps a weaker result holds: $\mathbb{E}^3\rightarrow (\ell_2,\ell_5)$.\\
The result (\ref{e:l2l4}) was generalised by Juh\'asz\cite{juhasz}, who proved that if $T_4$ is any configuration of $4$ points, then $\mathbb{E}^2\rightarrow (\ell_2,T_4)$. Juh\'asz (personal
communication, 10 February 2017) informed the author that Iv\'an's thesis~\cite{ivan} contains a proof that for any configuration $T_5$ of $5$ points, $\mathbb{E}^3\rightarrow (\ell_2,T_5)$ (which implies that $\mathbb{E}^3\rightarrow (\ell_2,\ell_5)$). Arman and Tsaturian\cite{at} proved that $\mathbb{E}^3\rightarrow (\ell_2,\ell_6)$.\\
In this paper, it is proved that $\mathbb{E}^2\rightarrow (\ell_2,\ell_5)$:
\begin{theorem}\label{main}
Let the Euclidean space $\mathbb{E}^2$ be coloured in red and blue so that there are no two red points distance $1$ apart. Then there exist five blue points that form an $\ell_5$.
\end{theorem}

\section{Proof of Theorem \ref{main}}

The proof is by contradiction; it is assumed that there are no five blue points forming an $\ell_5$. The following lemmas are needed.

\begin{lemma}\label{l:bluetr}
Let $\mathbb{E}^2$ be coloured in red and blue so that there is no red $\ell_2$. If there is no blue $\ell_5$, then there are no three blue points forming an equilateral triangle with side length $3$ and with a red centre.
\end{lemma}
\begin{proof}
Suppose that $\mathbb{E}^2$ is coloured in red and blue so that there is no red $\ell_2$ and no blue $\ell_5$.
Suppose that blue points $A$, $B$ and $C$ form an equilateral triangle with side length $3$ and with red centre $O$. Consider the part of the unit triangular lattice shown in Figure 1(a). The points $D$, $E$, $F$, $G$ are blue, since they are distance $1$ apart from $O$. The point $X$ is red; otherwise $XADEB$ is a red $\ell_5$. Similarly, $Y$ is red (to prevent red $YAFGC$). Then $X$ and $Y$ are two red points distance $1$ apart, which contradicts the assumption.
\end{proof}

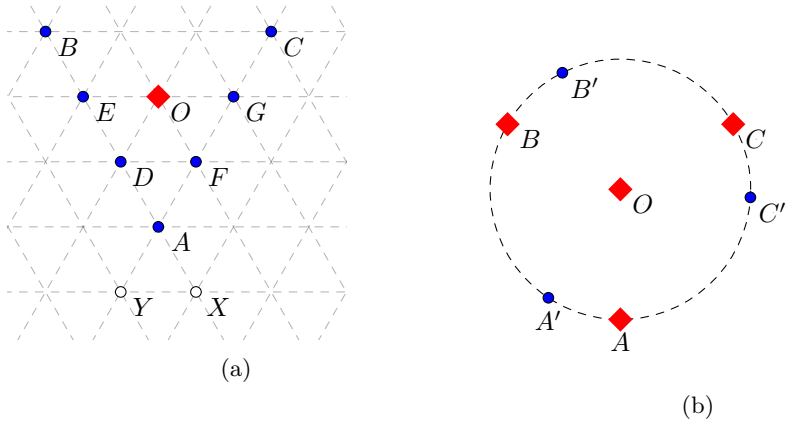
\begin{figure}[H]
    \begin{subfigure}{0.5\linewidth}

\begin{tikzpicture}[line cap=round,line join=round,>=triangle 45,x=1.0cm,y=1.0cm,scale=1]
\clip(-2,-1.5) rectangle (2.5,3);

\foreach \a in {1,2,...,10}
\draw [dashed, opacity=0.3] (-10,0)++(240:10)++(60:\a)++(120:\a)--+(32,0);

\foreach \a in {1,2,...,10}
\draw [dashed, opacity=0.3] (-10,0)++(240:10)++(60:\a+1)++(120:\a)--+(32,0);

\foreach \a in {1,2,...,30}
\draw [dashed, opacity=0.3] (-10,0)++(240:10)++(\a,0)--+(60:30);

\foreach \a in {1,2,...,30}
\draw [dashed, opacity=0.3] (-10,0)++(240:10)++(\a,0)--+(120:30);

\draw [fill=blue] (0,0) circle(2pt);
\draw [fill=blue] (0,0)++(0.3,-0.2) circle(0pt) node{$A$};

\draw [fill=blue] (0,0)++(60:1) circle(2pt);
\draw [fill=blue] (0,0)++(60:1)++(0.3,-0.2) circle(0pt) node{$F$};

\draw [fill=blue] (0,0)++(60:2) circle(2pt);
\draw [fill=blue] (0,0)++(60:2)++(0.3,-0.2) circle(0pt) node{$G$};

\draw [fill=blue] (0,0)++(60:3) circle(2pt);
\draw [fill=blue] (0,0)++(60:3)++(0.3,-0.2) circle(0pt) node{$C$};

\draw [fill=blue] (0,0)++(120:1) circle(2pt);
\draw [fill=blue] (0,0)++(120:1)++(0.3,-0.2) circle(0pt) node{$D$};

\draw [fill=blue] (0,0)++(120:2) circle(2pt);
\draw [fill=blue] (0,0)++(120:2)++(0.3,-0.2) circle(0pt) node{$E$};

\draw [fill=blue] (0,0)++(120:3) circle(2pt);
\draw [fill=blue] (0,0)++(120:3)++(0.3,-0.2) circle(0pt) node{$B$};

\draw [fill=red] (0,0)++(60:1)++(120:1) node[diamond, scale=0.7, fill=red] {};
\draw [fill=blue] (0,0)++(60:1)++(120:1)++(0.3,-0.2) circle(0pt) node{$O$};

\draw [fill=white] (0,0)++(240:1) circle(2pt);
\draw [fill=blue] (0,0)++(240:1)++(0.3,-0.2) circle(0pt) node{$Y$};

\draw [fill=white] (0,0)++(300:1) circle(2pt);
\draw [fill=blue] (0,0)++(300:1)++(0.3,-0.2) circle(0pt) node{$X$};

\end{tikzpicture}

\caption{}
\label{f:bluetr}
\end{subfigure}
\begin{subfigure}{0.5\linewidth}
        %

\begin{tikzpicture}[line cap=round,line join=round,>=triangle 45,x=1.0cm,y=1.0cm,scale=1]
\clip(-2,-2.5) rectangle (2.5,3);

\draw [fill=red] (0,0) node[diamond, scale=0.7, fill=red] {};
\draw [fill=blue] (0,0)++(0.3,-0.2) circle(0pt) node{$O$};

\draw [dashed, color=black] (0,0) circle (1.732cm);

\draw [fill=red] (0,0)++(30:1.732) node[diamond, scale=0.7, fill=red] {};
\draw [fill=blue] (0,0)++(30:1.732)++(0.3,-0.2) circle(0pt) node{$C$};

\draw [fill=blue] (0,0)++(-3.557:1.732) circle(2pt);
\draw [fill=blue] (0,0)++(-3.557:1.732)++(0.3,-0.2) circle(0pt) node{$C'$};

\draw [fill=red] (0,0)++(150:1.732) node[diamond, scale=0.7, fill=red] {};
\draw [fill=blue] (0,0)++(150:1.732)++(0.3,-0.2) circle(0pt) node{$B$};

\draw [fill=blue] (0,0)++(120-3.557:1.732) circle(2pt);
\draw [fill=blue] (0,0)++(120-3.557:1.732)++(0.3,-0.2) circle(0pt) node{$B'$};

\draw [fill=red] (0,0)++(270:1.732) node[diamond, scale=0.7, fill=red] {};
\draw [fill=blue] (0,0)++(270:1.732)++(0,-0.3) circle(0pt) node{$A$};

\draw [fill=blue] (0,0)++(240-3.557:1.732) circle(2pt);
\draw [fill=blue] (0,0)++(240-3.557:1.732)++(0,-0.3) circle(0pt) node{$A'$};

\end{tikzpicture}

\caption{}\label{f:redtr}

\end{subfigure}

\caption{Red points are denoted by diamonds, blue points are denoted by discs.}

\end{figure}

\begin{lemma}\label{l:redtr}
Let $\mathbb{E}^2$ be coloured in red and blue so that there is no red $\ell_2$. If there is no blue $\ell_5$, then there are no three red points forming an equilateral triangle with side length $\sqrt{3}$ and with a red centre.
\end{lemma}
\begin{proof}
Suppose that $\mathbb{E}^2$ is coloured in red and blue so that there is no red $\ell_2$ and no blue $\ell_5$.
Suppose that blue points $A$, $B$ and $C$ form an equilateral triangle with side length $3$ and with red centre $O$. Let $A'$, $B'$, $C'$ be the images of $A$, $B$ and $C$, respectively, under a rotation about $O$ so that $AA'=BB'=CC'=1$ (see Figure 1(b)). Then $A'$, $B'$, $C'$ are blue and form an equilateral triangle with side length $\sqrt{3}$ and red center $O$, which contradicts the result of Lemma~\ref{l:bluetr}.
\end{proof}

Define $\mathfrak{T}_3$, $\mathfrak{T}_4$, $\mathfrak{T}_5$, $\mathfrak{T}_6$, $\mathfrak{T}_7$ to be the configurations of three, four, five, six and seven points (respectively), depicted in Figure~\ref{f:figures} (all the smallest distances between the points are equal to $\sqrt{3}$). 

\begin{figure}[H]
\begin{center}
\begin{tikzpicture}[line cap=round,line join=round,>=triangle 45,x=1.0cm,y=1.0cm,scale=0.5]
\clip(-3,-7) rectangle (16,3.5);

\foreach \a in {1,2,...,10}
\draw [dashed, opacity=0.2] (-10,0)++(240:10)++(60:\a)++(120:\a)--+(32,0);

\foreach \a in {1,2,...,10}
\draw [dashed, opacity=0.2] (-10,0)++(240:10)++(60:\a+1)++(120:\a)--+(32,0);

\foreach \a in {1,2,...,30}
\draw [dashed, opacity=0.2] (-10,0)++(240:10)++(\a,0)--+(60:30);

\foreach \a in {1,2,...,40}
\draw [dashed, opacity=0.2] (-10,0)++(240:10)++(\a,0)--+(120:30);

\draw [fill=red] (0,0) node[diamond, scale=0.7, fill=red] {};
\draw [fill=red] (-1,0)++(120:1) node[diamond, scale=0.7, fill=red] {};
\draw [fill=red] (-1,0)++(60:2) node[diamond, scale=0.7, fill=red] {};
\draw [fill=blue] (0,0)++(60:-1) circle(0pt) node{$\mathfrak{T}_3$};

\draw [fill=red] (5,0) node[diamond, scale=0.7, fill=red] {};
\draw [fill=red] (4,0)++(120:1) node[diamond, scale=0.7, fill=red] {};
\draw [fill=red] (4,0)++(60:2) node[diamond, scale=0.7, fill=red] {};
\draw [fill=red] (6,0)++(60:1) node[diamond, scale=0.7, fill=red] {};
\draw [fill=blue] (5.5,0)++(60:-1) circle(0pt) node{$\mathfrak{T}_4$};

\draw [fill=red] (12,0) node[diamond, scale=0.7, fill=red] {};
\draw [fill=red] (11,0)++(120:1) node[diamond, scale=0.7, fill=red] {};
\draw [fill=red] (11,0)++(60:2) node[diamond, scale=0.7, fill=red] {};
\draw [fill=red] (13,0)++(60:1) node[diamond, scale=0.7, fill=red] {};
\draw [fill=red] (12,0)++(120:3) node[diamond, scale=0.7, fill=red] {};
\draw [fill=blue] (12.5,0)++(60:-1) circle(0pt) node{$\mathfrak{T}_5$};

\draw [fill=red] (2,0)++(60:-4) node[diamond, scale=0.7, fill=red] {};
\draw [fill=red] (3,0)++(60:-3) node[diamond, scale=0.7, fill=red] {};
\draw [fill=red] (4,0)++(60:-2) node[diamond, scale=0.7, fill=red] {};
\draw [fill=red] (4,0)++(60:-5) node[diamond, scale=0.7, fill=red] {};
\draw [fill=red] (5,0)++(60:-4) node[diamond, scale=0.7, fill=red] {};
\draw [fill=red] (6,0)++(60:-6) node[diamond, scale=0.7, fill=red] {};
\draw [fill=blue] (5,0)++(60:-7) circle(0pt) node{$\mathfrak{T}_6$};

\draw [fill=red] (9,0)++(60:-2) node[diamond, scale=0.7, fill=red] {};
\draw [fill=red] (10,0)++(60:-4) node[diamond, scale=0.7, fill=red] {};
\draw [fill=red] (11,0)++(60:-3) node[diamond, scale=0.7, fill=red] {};
\draw [fill=red] (12,0)++(60:-5) node[diamond, scale=0.7, fill=red] {};
\draw [fill=red] (13,0)++(60:-4) node[diamond, scale=0.7, fill=red] {};
\draw [fill=red] (14,0)++(60:-6) node[diamond, scale=0.7, fill=red] {};
\draw [fill=red] (15,0)++(60:-5) node[diamond, scale=0.7, fill=red] {};
\draw [fill=blue] (13,0)++(60:-7) circle(0pt) node{$\mathfrak{T}_7$};

\end{tikzpicture}
\end{center}
\caption{}\label{f:figures}
\end{figure}

\begin{lemma}\label{l:t7}
Let $\mathbb{E}^2$ be coloured in red and blue so that there is no red $\ell_2$. If there is no blue $\ell_5$, then there are no seven red points forming a $\mathfrak{T}_7$.
\end{lemma}
\begin{proof}
Suppose that $\mathbb{E}^2$ is coloured in red and blue so that there is no red $\ell_2$ and no blue $\ell_5$.
Suppose that $A$, $B$, $C$, $D$, $E$, $F$ and $G$ are red points forming a $\mathfrak{T}_7$ (as in Figure~\ref{f:t7}).
Let $X$ be the reflection of $F$ in $BC$.
Let $X'$, $A'$, $F'$ be the images of $X$, $A$, $F$, respectively, under the clockwise rotation about $B$ such that $XX'=AA'=FF'=1$. Since $A$ and $F$ are red, $A'$ and $F'$ are blue. If $X'$ is blue, then $X'A'F'$ is a blue equilateral triangle with side length $3$ and red center $B$, which contradicts the result of Lemma~\ref{l:bluetr}. Therefore, $X'$ is red.
Let $X''$, $D''$, $F''$ be the images of $X$, $D$, $F$, respectively, under the clockwise rotation about $C$ such that $XX''=DD''=FF''=1$. Since $D$ and $F$ are red, $D''$ and $F''$ are blue. If $X''$ is blue, then $X''D''F''$ is a blue equilateral triangle with side length $3$ and red center $C$, which contradicts the result of Lemma~\ref{l:bluetr}. Therefore, $X''$ is red.
Since $X'$ can be obtained from $X''$ by the clockwise rotation through $60\degree$ about $X$, $XX'X''$ is a unit equilateral triangle, hence $X'X''$ is a red $\ell_2$, which contradicts the assumption of the lemma.
\end{proof}

\begin{figure}[H]
\begin{center}
\begin{tikzpicture}[line cap=round,line join=round,>=triangle 45,x=1.0cm,y=1.0cm,scale=1]
\clip(-2,-2.5) rectangle (4,3);

\draw [dashed, color=black] (0,0) circle (1.732cm);
\draw [dashed, color=black] (0,0)++(30:1)++(-30:1) circle (1.732cm);

\draw [fill=red] (0,0) node[diamond, scale=0.7, fill=red] {};
\draw [fill=blue] (0,0)++(0.3,-0.2) circle(0pt) node{$B$};

\draw [fill=red] (1.732,0) node[diamond, scale=0.7, fill=red] {};
\draw [fill=blue] (1.732,0)++(0.3,-0.2) circle(0pt) node{$C$};

\draw [fill=red] (-1.732,0) node[diamond, scale=0.7, fill=red] {};
\draw [fill=blue] (-1.732,0)++(0.3,-0.2) circle(0pt) node{$A$};

\draw [fill=blue] (180-33.557:1.732) circle(2pt);
\draw [fill=blue] (180-33.557:1.732)++(0.3,-0.2) circle(0pt) node{$A'$};

\draw [fill=red] (3.464,0) node[diamond, scale=0.7, fill=red] {};
\draw [fill=blue] (3.464,0)++(0.3,-0.2) circle(0pt) node{$D$};

\draw [fill=blue] (60:1.732) circle(2pt);
\draw [fill=blue] (60:1.732)++(0.1,-0.3) circle(0pt) node{$X$};

\draw [fill=blue] (60-33.557:1.732) node[diamond, scale=0.7, fill=red] {};
\draw [fill=blue] (60-33.557:1.732)++(0.3,-0.2) circle(0pt) node{$X'$};

\draw [fill=blue] (-60:1.732) node[diamond, scale=0.7, fill=red] {};
\draw [fill=blue] (-60:1.732)++(0,-0.3) circle(0pt) node{$F$};

\draw [fill=blue] (-60-33.557:1.732) circle(2pt);
\draw [fill=blue] (-60-33.557:1.732)++(0,-0.3) circle(0pt) node{$F'$};

\draw [fill=blue] (-120:1.732) node[diamond, scale=0.7, fill=red] {};
\draw [fill=blue] (-120:1.732)++(-0.3,-0.2) circle(0pt) node{$E$};

\draw [fill=blue] (1.732,0)++(-60:1.732) node[diamond, scale=0.7, fill=red] {};
\draw [fill=blue] (1.732,0)++(-60:1.732)++(0.3,-0.2) circle(0pt) node{$G$};

\draw [fill=blue] (1.732,0)++(-33.557:1.732) circle(2pt);
\draw [fill=blue] (1.732,0)++(-33.557:1.732)++(0.3,-0.2) circle(0pt) node{$D''$};

\draw [fill=blue] (1.732,0)++(120-33.557:1.732) node[diamond, scale=0.7, fill=red] {};
\draw [fill=blue] (1.732,0)++(120-33.557:1.732)++(0.3,-0.3) circle(0pt) node{$X''$};

\draw [fill=blue] (1.732,0)++(-120-33.557:1.732) circle(2pt);
\draw [fill=blue] (1.732,0)++(-120-33.557:1.732)++(0.4,-0.1) circle(0pt) node{$F''$};

\end{tikzpicture}
\end{center}
\caption{}\label{f:t7}
\end{figure}

\begin{lemma}\label{l:t3t6}
Let $\mathbb{E}^2$ be coloured in red and blue so that there is no red $\ell_2$. Let $A$, $B$, $C$ be three red points forming a $\mathfrak{T}_3$. If there is no blue $\ell_5$, then there exists a red $\mathfrak{T}_6$ that contains $\{A,B,C\}$ as a subset.
\end{lemma}
\begin{proof}
Suppose that $\mathbb{E}^2$ is coloured in red and blue so that there is no red $\ell_2$ and no blue $\ell_5$.
Let $A$, $B$, $C$ be three red points forming a $\mathfrak{T}_3$. Consider the unit triangular lattice depicted in Figure~\ref{f:t3t4}.\\
Suppose that there is no red point $D$ such that $A$, $B$, $C$, $D$ form a $\mathfrak{T}_4$. Then points $X$, $Y$, $Z$ are blue. Points $E$, $F$, $G$, $H$, $I$, $J$ are blue, since each of them is distance $1$ apart from a red point. If the point $K$ is red, then the points $L$ and $M$ are blue and $LMYGH$ is a blue $\ell_5$. Therefore, $K$ is blue. Then $N$ is red (otherwise $KJIZN$ is a blue $\ell_5$), hence $P$ and $Q$ are blue, which leads to a blue $\ell_5$ $PQFEX$. A contradiction is obtained, therefore there exists a red point $D$ such that $A$, $B$, $C$, $D$ form a $\mathfrak{T}_4$.\\

\begin{figure}[H]
\begin{center}
\begin{tikzpicture}[line cap=round,line join=round,>=triangle 45,x=1.0cm,y=1.0cm,scale=0.7]
\clip(-2.5,-2.5) rectangle (4,4);

\foreach \a in {1,2,...,10}
\draw [dashed, opacity=0.3] (-10,0)++(240:10)++(60:\a)++(120:\a)--+(32,0);

\foreach \a in {1,2,...,10}
\draw [dashed, opacity=0.3] (-10,0)++(240:10)++(60:\a+1)++(120:\a)--+(32,0);

\foreach \a in {1,2,...,30}
\draw [dashed, opacity=0.3] (-10,0)++(240:10)++(\a,0)--+(60:30);

\foreach \a in {1,2,...,30}
\draw [dashed, opacity=0.3] (-10,0)++(240:10)++(\a,0)--+(120:30);

\draw [fill=red] (0,0) node[diamond, scale=0.7, fill=red] {};
\draw [fill=blue] (0,0)++(0.3,-0.2) circle(0pt) node{$A$};

\draw [fill=red] (1,0)++(60:1) node[diamond, scale=0.7, fill=red] {};
\draw [fill=blue] (1,0)++(60:1)++(0.3,-0.2) circle(0pt) node{$B$};

\draw [fill=red] (2,0)++(60:-1) node[diamond, scale=0.7, fill=red] {};
\draw [fill=blue] (2,0)++(60:-1)++(0.3,-0.2) circle(0pt) node{$C$};

\draw [fill=blue] (-1,0) circle(2pt);
\draw [fill=blue] (-1,0)++(0.3,-0.2) circle(0pt) node{$J$};

\draw [fill=blue] (-1,0)++(60:1) circle(2pt);
\draw [fill=blue] (-1,0)++(60:1)++(0.3,-0.2) circle(0pt) node{$I$};

\draw [fill=blue] (-1,0)++(60:2) circle(2pt);
\draw [fill=blue] (-1,0)++(60:2)++(0.3,-0.2) circle(0pt) node{$Z$};

\draw [fill=white] (-1,0)++(60:3) circle(2pt);
\draw [fill=blue] (-1,0)++(60:3)++(0.3,-0.2) circle(0pt) node{$N$};

\draw [fill=white] (-1,0)++(60:4) circle(2pt);
\draw [fill=blue] (-1,0)++(60:4)++(0.3,-0.2) circle(0pt) node{$P$};

\draw [fill=white] (-1,0)++(60:-1) circle(2pt);
\draw [fill=blue] (-1,0)++(60:-1)++(0.3,-0.2) circle(0pt) node{$K$};

\draw [fill=white] (-1,0)++(60:-2) circle(2pt);
\draw [fill=blue] (-1,0)++(60:-2)++(0.3,-0.2) circle(0pt) node{$L$};

\draw [fill=white] (0,0)++(60:-2) circle(2pt);
\draw [fill=blue] (0,0)++(60:-2)++(0.3,-0.2) circle(0pt) node{$M$};

\draw [fill=blue] (1,0)++(60:-2) circle(2pt);
\draw [fill=blue] (1,0)++(60:-2)++(0.3,-0.2) circle(0pt) node{$Y$};

\draw [fill=blue] (2,0)++(60:-2) circle(2pt);
\draw [fill=blue] (2,0)++(60:-2)++(0.3,-0.2) circle(0pt) node{$G$};

\draw [fill=blue] (3,0)++(60:-2) circle(2pt);
\draw [fill=blue] (3,0)++(60:-2)++(0.3,-0.2) circle(0pt) node{$H$};

\draw [fill=blue] (3,0) circle(2pt);
\draw [fill=blue] (3,0)++(0.3,-0.2) circle(0pt) node{$X$};

\draw [fill=blue] (3,0)++(120:1) circle(2pt);
\draw [fill=blue] (3,0)++(120:1)++(0.3,-0.2) circle(0pt) node{$E$};

\draw [fill=blue] (3,0)++(120:2) circle(2pt);
\draw [fill=blue] (3,0)++(120:2)++(0.3,-0.2) circle(0pt) node{$F$};

\draw [fill=white] (3,0)++(120:3) circle(2pt);
\draw [fill=blue] (3,0)++(120:3)++(0.3,-0.2) circle(0pt) node{$Q$};

\end{tikzpicture}
\end{center}
\caption{}\label{f:t3t4}
\end{figure}

Let $A$, $B$, $C$, $D$ form a red $\mathfrak{T}_4$. Consider the part of the unit triangular lattice depicted in Figure~\ref{f:t4t5}. Suppose that there is no red point $E$ such that $A$, $B$, $C$, $D$, $E$ form a $\mathfrak{T}_5$. Then the points $X$, $F$ and $G$ are blue. Points $H$, $I$, $K$, $L$, $M$, $N$ are blue, since each of them is distance $1$ apart from a red point. Point $P$ is red (otherwise $FHIGP$ is a blue $\ell_5$), therefore $Q$ and $R$ are blue. Then $X$, $N$, $M$, $Q$, $R$ form a blue $\ell_5$, which gives a contradiction. Hence, there exists a red point $E$ such that $A$, $B$, $C$, $D$, $E$ form a $\mathfrak{T}_5$.\\

\begin{figure}[H]
\begin{center}
\begin{tikzpicture}[line cap=round,line join=round,>=triangle 45,x=1.0cm,y=1.0cm,scale=0.7]
\clip(-0.5,-2.5) rectangle (6,2.5);

\foreach \a in {1,2,...,10}
\draw [dashed, opacity=0.3] (-10,0)++(240:10)++(60:\a)++(120:\a)--+(32,0);

\foreach \a in {1,2,...,10}
\draw [dashed, opacity=0.3] (-10,0)++(240:10)++(60:\a+1)++(120:\a)--+(32,0);

\foreach \a in {1,2,...,30}
\draw [dashed, opacity=0.3] (-10,0)++(240:10)++(\a,0)--+(60:30);

\foreach \a in {1,2,...,30}
\draw [dashed, opacity=0.3] (-10,0)++(240:10)++(\a,0)--+(120:30);

\draw [fill=red] (0,0) node[diamond, scale=0.7, fill=red] {};
\draw [fill=blue] (0,0)++(0.3,-0.2) circle(0pt) node{$A$};

\draw [fill=red] (1,0)++(60:1) node[diamond, scale=0.7, fill=red] {};
\draw [fill=blue] (1,0)++(60:1)++(0.3,-0.2) circle(0pt) node{$B$};

\draw [fill=red] (2,0)++(60:-1) node[diamond, scale=0.7, fill=red] {};
\draw [fill=blue] (2,0)++(60:-1)++(0.3,-0.2) circle(0pt) node{$C$};

\draw [fill=red] (3,0) node[diamond, scale=0.7, fill=red] {};
\draw [fill=blue] (3,0)++(0.3,-0.2) circle(0pt) node{$D$};

\draw [fill=blue] (0,0)++(60:2) circle(2pt);
\draw [fill=blue] (0,0)++(60:2)++(0.3,-0.2) circle(0pt) node{$K$};

\draw [fill=blue] (1,0)++(60:2) circle(2pt);
\draw [fill=blue] (1,0)++(60:2)++(0.3,-0.2) circle(0pt) node{$L$};

\draw [fill=blue] (2,0)++(60:2) circle(2pt);
\draw [fill=blue] (2,0)++(60:2)++(0.3,-0.2) circle(0pt) node{$X$};

\draw [fill=blue] (1,0)++(60:-2) circle(2pt);
\draw [fill=blue] (1,0)++(60:-2)++(0.3,-0.2) circle(0pt) node{$F$};

\draw [fill=blue] (2,0)++(60:-2) circle(2pt);
\draw [fill=blue] (2,0)++(60:-2)++(0.3,-0.2) circle(0pt) node{$H$};

\draw [fill=blue] (3,0)++(60:-2) circle(2pt);
\draw [fill=blue] (3,0)++(60:-2)++(0.3,-0.2) circle(0pt) node{$I$};

\draw [fill=blue] (4,0)++(60:-2) circle(2pt);
\draw [fill=blue] (4,0)++(60:-2)++(0.3,-0.2) circle(0pt) node{$G$};

\draw [fill=white] (5,0)++(60:-2) circle(2pt);
\draw [fill=blue] (5,0)++(60:-2)++(0.3,-0.2) circle(0pt) node{$P$};

\draw [fill=white] (6,0)++(60:-2) circle(2pt);
\draw [fill=blue] (6,0)++(60:-2)++(0.3,-0.2) circle(0pt) node{$R$};

\draw [fill=blue] (4,0) circle(2pt);
\draw [fill=blue] (4,0)++(0.3,-0.2) circle(0pt) node{$M$};

\draw [fill=blue] (4,0)++(120:1) circle(2pt);
\draw [fill=blue] (4,0)++(120:1)++(0.3,-0.2) circle(0pt) node{$N$};

\draw [fill=white] (4,0)++(120:-1) circle(2pt);
\draw [fill=blue] (4,0)++(120:-1)++(0.3,-0.2) circle(0pt) node{$Q$};

\end{tikzpicture}
\end{center}
\caption{}\label{f:t4t5}
\end{figure}

Let $A$, $B$, $C$, $D$, $E$ form a $\mathfrak{T}_5$ (Figure~\ref{f:t5t6}). Suppose that $F$ is blue. By Lemma~\ref{l:redtr}, points $X$ and $Y$ are blue (otherwise $X$, $E$, $C$ ($Y$, $A$, $D$) form a red triangle with side length $3$ and red center $B$). Points $G$, $H$, $I$, $J$, $K$, $L$, $M$, $N$ are blue, since each one of them is at distance $1$ from a red point. If point $P$ is blue, then $Q$ is red (otherwise $QPKLF$ is a blue $\ell_5$), $U$ and $T$ are blue and form a blue $\ell_5$ with points $G$, $H$ and $X$. Therefore, $P$ is red. Similarly, $R$ is red (otherwise $S$ is red and $VWJIY$ is a blue $\ell_5$). Then $A$, $B$, $C$, $D$, $E$, $P$ and $R$ form a red $\mathfrak{T}_7$, which is not possible by Lemma~\ref{l:t7}. Therefore, $F$ is red and $A$, $B$, $C$, $D$, $E$, $F$ form a red $\mathfrak{T}_6$.

\begin{figure}[H]
\begin{center}
\begin{tikzpicture}[line cap=round,line join=round,>=triangle 45,x=1.0cm,y=1.0cm,scale=0.7]
\clip(-2.5,-2.5) rectangle (6.5,3);

\foreach \a in {1,2,...,10}
\draw [dashed, opacity=0.3] (-10,0)++(240:10)++(60:\a)++(120:\a)--+(32,0);

\foreach \a in {1,2,...,10}
\draw [dashed, opacity=0.3] (-10,0)++(240:10)++(60:\a+1)++(120:\a)--+(32,0);

\foreach \a in {1,2,...,30}
\draw [dashed, opacity=0.3] (-10,0)++(240:10)++(\a,0)--+(60:30);

\foreach \a in {1,2,...,30}
\draw [dashed, opacity=0.3] (-10,0)++(240:10)++(\a,0)--+(120:30);

\draw [fill=red] (0,0) node[diamond, scale=0.7, fill=red] {};
\draw [fill=blue] (0,0)++(0.3,-0.2) circle(0pt) node{$A$};

\draw [fill=red] (1,0)++(60:1) node[diamond, scale=0.7, fill=red] {};
\draw [fill=blue] (1,0)++(60:1)++(0.3,-0.2) circle(0pt) node{$B$};

\draw [fill=red] (2,0)++(60:-1) node[diamond, scale=0.7, fill=red] {};
\draw [fill=blue] (2,0)++(60:-1)++(0.3,-0.2) circle(0pt) node{$C$};

\draw [fill=red] (3,0) node[diamond, scale=0.7, fill=red] {};
\draw [fill=blue] (3,0)++(0.3,-0.2) circle(0pt) node{$D$};

\draw [fill=red] (2,0)++(60:2) node[diamond, scale=0.7, fill=red] {};
\draw [fill=blue] (2,0)++(60:2)++(0.3,-0.2) circle(0pt) node{$E$};

\draw [fill=blue] (-1,0) circle(2pt);
\draw [fill=blue] (-1,0)++(0.3,-0.2) circle(0pt) node{$G$};

\draw [fill=blue] (-1,0)++(60:1) circle(2pt);
\draw [fill=blue] (-1,0)++(60:1)++(0.3,-0.2) circle(0pt) node{$H$};

\draw [fill=blue] (-1,0)++(60:2) circle(2pt);
\draw [fill=blue] (-1,0)++(60:2)++(0.3,-0.2) circle(0pt) node{$X$};

\draw [fill=white] (-1,0)++(60:-1) circle(2pt);
\draw [fill=blue] (-1,0)++(60:-1)++(0.3,-0.2) circle(0pt) node{$U$};

\draw [fill=white] (-1,0)++(60:-2) circle(2pt);
\draw [fill=blue] (-1,0)++(60:-2)++(0.3,-0.2) circle(0pt) node{$T$};

\draw [fill=white] (0,0)++(60:-2) circle(2pt);
\draw [fill=blue] (0,0)++(60:-2)++(0.3,-0.2) circle(0pt) node{$Q$};

\draw [fill=white] (1,0)++(60:-2) circle(2pt);
\draw [fill=blue] (1,0)++(60:-2)++(0.3,-0.2) circle(0pt) node{$P$};

\draw [fill=blue] (2,0)++(60:-2) circle(2pt);
\draw [fill=blue] (2,0)++(60:-2)++(0.3,-0.2) circle(0pt) node{$K$};

\draw [fill=blue] (3,0)++(60:-2) circle(2pt);
\draw [fill=blue] (3,0)++(60:-2)++(0.3,-0.2) circle(0pt) node{$L$};

\draw [fill=blue] (4,0)++(60:-2) circle(2pt);
\draw [fill=blue] (4,0)++(60:-2)++(0.3,-0.2) circle(0pt) node{$F$};

\draw [fill=blue] (4,0) circle(2pt);
\draw [fill=blue] (4,0)++(0.3,-0.2) circle(0pt) node{$N$};

\draw [fill=blue] (4,0)++(60:-1) circle(2pt);
\draw [fill=blue] (4,0)++(60:-1)++(0.3,-0.2) circle(0pt) node{$M$};

\draw [fill=white] (4,0)++(60:1) circle(2pt);
\draw [fill=blue] (4,0)++(60:1)++(0.3,-0.2) circle(0pt) node{$R$};

\draw [fill=white] (4,0)++(60:2) circle(2pt);
\draw [fill=blue] (4,0)++(60:2)++(0.3,-0.2) circle(0pt) node{$S$};

\draw [fill=white] (4,0)++(60:3) circle(2pt);
\draw [fill=blue] (4,0)++(60:3)++(0.3,-0.2) circle(0pt) node{$V$};

\draw [fill=white] (3,0)++(60:3) circle(2pt);
\draw [fill=blue] (3,0)++(60:3)++(0.3,-0.2) circle(0pt) node{$W$};

\draw [fill=blue] (2,0)++(60:3) circle(2pt);
\draw [fill=blue] (2,0)++(60:3)++(0.3,-0.2) circle(0pt) node{$J$};

\draw [fill=blue] (1,0)++(60:3) circle(2pt);
\draw [fill=blue] (1,0)++(60:3)++(0.3,-0.2) circle(0pt) node{$I$};

\draw [fill=blue] (0,0)++(60:3) circle(2pt);
\draw [fill=blue] (0,0)++(60:3)++(0.3,-0.2) circle(0pt) node{$Y$};

\end{tikzpicture}
\end{center}
\caption{}\label{f:t5t6}
\end{figure}

\end{proof}

\begin{lemma}\label{l:col1}
Let $\mathbb{E}^2$ be coloured in red and blue so that there is no red $\ell_2$. Let $\mathfrak{L}$ be a unit triangular lattice that contains three red points forming a $\mathfrak{T}_3$. If there is no blue $\ell_5$, then the colouring of $\mathfrak{L}$ is unique (up to translation or rotation by a multiple of $60\degree$), and is depicted in Figure~\ref{f:tcol}.
\end{lemma}

\begin{figure}[H]
\begin{center}
\begin{tikzpicture}[line cap=round,line join=round,>=triangle 45,x=1.0cm,y=1.0cm,scale=0.6]
\clip(4.5,6.5) rectangle (23.5,19.5);

\foreach \a in {0,1,...,15}
\draw [dashed, opacity=0.5] (-10,0)++(60:\a)++(120:\a)--+(40,0);

\foreach \a in {0,1,...,15}
\draw [dashed, opacity=0.5] (-10,0)++(60:\a+1)++(120:\a)--+(40,0);

\foreach \a in {0,1,...,40}
\draw [dashed, opacity=0.5] (-10,0)++(\a,0)--+(60:30);

\foreach \a in {0,1,...,45}
\draw [dashed, opacity=0.5] (-10,0)++(\a,0)--+(120:30);

\foreach \b in {0,1,...,30}
\foreach \a in {0,1,...,30}
\draw [fill=blue] (-5,0)++(\a,0.)++(60:\b)++(120:\b) circle (2pt);

\foreach \b in {0,1,...,30}
\foreach \a in {0,1,...,30}
\draw [fill=blue] (-5,0)++(\a,0)++(60:1)++(60:\b)++(120:\b) circle (2pt);

\foreach \a in {-1,...,5}
\foreach \b in {0,...,10}
\draw [fill=red] (5*\a,0)++(60:5*\b) node[diamond, scale=0.7, fill=red] {};

\foreach \a in {-1,...,5}
\foreach \b in {0,...,10}
\draw [fill=red] (5*\a+1,0)++(60:1)++(60:5*\b) node[diamond, scale=0.7, fill=red] {};

\foreach \a in {-1,...,5}
\foreach \b in {0,...,10}
\draw [fill=red] (5*\a+2,0)++(60:-1)++(60:5*\b) node[diamond, scale=0.7, fill=red] {};

\foreach \a in {-1,...,5}
\foreach \b in {0,...,10}
\draw [fill=red] (5*\a+2,0)++(60:2)++(60:5*\b) node[diamond, scale=0.7, fill=red] {};

\foreach \a in {-1,...,5}
\foreach \b in {0,...,10}
\draw [fill=red] (5*\a+3,0)++(60:5*\b) node[diamond, scale=0.7, fill=red] {};

\foreach \a in {-1,...,5}
\foreach \b in {0,...,10}
\draw [fill=red] (5*\a+4,0)++(60:-2)++(60:5*\b) node[diamond, scale=0.7, fill=red] {};

\end{tikzpicture}
\end{center}
\caption{}\label{f:tcol}
\end{figure}

\begin{proof}
Suppose that $\mathbb{E}^2$ is coloured in red and blue so that there is no red $\ell_2$ and no blue $\ell_5$.
Suppose there exist three red points of $\mathfrak{L}$ that form a $\mathfrak{T}_3$. By Lemma~\ref{l:t3t6}, it may be assumed that there is a red $\mathfrak{T}_6$. Denote its points by $A$, $B$, $C$, $D$, $E$, $F$ (see Figure~\ref{f:tcol1}). It will be proved that the translate $A'B'C'D'E'F'$ of $ABCDEF$ by the vector of length $5$ collinear to $\overrightarrow{AD}$ is red.\\
Consider the points shown in Figure~\ref{f:tcol1}. Since $A$, $D$ and $F$ are red, by Lemma~\ref{l:redtr}, $I$ is blue. Since $C$, $F$ and $D$ are red, by Lemma~\ref{l:redtr}, $J$ is blue. Points $K$, $L$, $M$, $N$ are blue, since each one is distance $1$ apart from a red point. If $R$ is red, then both $P$ and $Q$ are blue and form a blue $\ell_5$ with $K$, $L$ and $I$. Therefore $R$ is blue. Then the point $A'$ is red (otherwise $A'JNMR$ is a red $\ell_5$).\\
Since $S_1$, $S_2$, $S_3$, $S_4$ are blue (as distance $1$ apart from red points $D$ and $A'$), $B'$ is red. Similarly, $F'$ is red. Points $V$ and $W$ are blue as they are distance $1$ apart from $C$. Points $U$ is blue by Lemma~\ref{l:redtr} (since $A$, $D$ and $B$ are red). If $X$ is red, then $X_1$ and $X_2$ are blue and a blue $\ell_5$ $UVWX_1X_2$ is formed. Therefore, $X$ is blue. Similarly, $Y$ is blue. By Lemma~\ref{l:t3t6}, $A'B'F'$ must be contained in a red $\mathfrak{T}_6$, and since $X$ and $Y$ are blue, the only possible such $\mathfrak{T}_6$ is $A'B'C'D'E'F'$. Hence, $A'$, $B'$, $C'$, $D'$, $E'$, $F'$ are blue.\\
Similarly, the translates of $ABCDEF$ by vectors of length $5$ collinear to $\overrightarrow{EB}$ and $\overrightarrow{CF}$ are red. By repeatedly applying the same argument to the new red translates, it can be seen that all the translates of $ABCDEF$ by a multiple of $5$ in 
$\mathfrak{L}$ are red. All the other points are blue, as each one is distance $1$ apart from a red point. Hence, the colouring as in Figure~\ref{f:tcol} is obtained.
\end{proof}

\begin{figure}[H]
\begin{center}
\begin{tikzpicture}[line cap=round,line join=round,>=triangle 45,x=1.0cm,y=1.0cm,scale=0.9]
\clip(-0.5,-5) rectangle (9,3);

\foreach \a in {1,2,...,10}
\draw [dashed, opacity=0.3] (-10,0)++(240:10)++(60:\a)++(120:\a)--+(32,0);

\foreach \a in {1,2,...,10}
\draw [dashed, opacity=0.3] (-10,0)++(240:10)++(60:\a+1)++(120:\a)--+(32,0);

\foreach \a in {1,2,...,30}
\draw [dashed, opacity=0.3] (-10,0)++(240:10)++(\a,0)--+(60:30);

\foreach \a in {1,2,...,40}
\draw [dashed, opacity=0.3] (-10,0)++(240:10)++(\a,0)--+(120:30);

\draw [fill=red] (0,0) node[diamond, scale=0.7, fill=red] {};
\draw [fill=blue] (0,0)++(0.3,-0.2) circle(0pt) node{$A$};

\draw [fill=red] (1,0)++(60:1) node[diamond, scale=0.7, fill=red] {};
\draw [fill=blue] (1,0)++(60:1)++(0.3,-0.2) circle(0pt) node{$B$};

\draw [fill=red] (2,0)++(60:-1) node[diamond, scale=0.7, fill=red] {};
\draw [fill=blue] (2,0)++(60:-1)++(0.3,-0.2) circle(0pt) node{$F$};

\draw [fill=red] (2,0)++(60:2) node[diamond, scale=0.7, fill=red] {};
\draw [fill=blue] (2,0)++(60:2)++(0.3,-0.2) circle(0pt) node{$C$};

\draw [fill=red] (3,0) node[diamond, scale=0.7, fill=red] {};
\draw [fill=blue] (3,0)++(0.3,-0.2) circle(0pt) node{$D$};

\draw [fill=red] (4,0)++(60:-2) node[diamond, scale=0.7, fill=red] {};
\draw [fill=blue] (4,0)++(60:-2)++(0.3,-0.2) circle(0pt) node{$E$};

\draw [fill=white] (5,0) circle(2pt);
\draw [fill=blue] (5,0)++(0.3,-0.2) circle(0pt) node{$A'$};

\draw [fill=white] (6,0)++(60:1) circle(2pt);
\draw [fill=blue] (6,0)++(60:1)++(0.3,-0.2) circle(0pt) node{$B'$};

\draw [fill=white] (7,0)++(60:-1) circle(2pt);
\draw [fill=blue] (7,0)++(60:-1)++(0.3,-0.2) circle(0pt) node{$F'$};

\draw [fill=white] (7,0)++(60:2) circle(2pt);
\draw [fill=blue] (7,0)++(60:2)++(0.3,-0.2) circle(0pt) node{$C'$};

\draw [fill=white] (8,0) circle(2pt);
\draw [fill=blue] (8,0)++(0.3,-0.2) circle(0pt) node{$D'$};

\draw [fill=white] (9,0)++(60:-2) circle(2pt);
\draw [fill=blue] (9,0)++(60:-2)++(0.3,-0.2) circle(0pt) node{$E'$};

\draw [fill=blue] (0,0)++(-60:1) circle(2pt);
\draw [fill=blue] (0,0)++(-60:1)++(0.3,-0.2) circle(0pt) node{$K$};

\draw [fill=blue] (0,0)++(-60:2) circle(2pt);
\draw [fill=blue] (0,0)++(-60:2)++(0.3,-0.2) circle(0pt) node{$L$};

\draw [fill=blue] (0,0)++(-60:3) circle(2pt);
\draw [fill=blue] (0,0)++(-60:3)++(0.3,-0.2) circle(0pt) node{$I$};

\draw [fill=white] (0,0)++(-60:4) circle(2pt);
\draw [fill=blue] (0,0)++(-60:4)++(0.3,-0.2) circle(0pt) node{$Q$};

\draw [fill=white] (0,0)++(-60:5) circle(2pt);
\draw [fill=blue] (0,0)++(-60:5)++(0.3,-0.2) circle(0pt) node{$P$};

\draw [fill=blue] (5,0)++(60:-1) circle(2pt);
\draw [fill=blue] (5,0)++(60:-1)++(0.3,-0.2) circle(0pt) node{$J$};

\draw [fill=blue] (5,0)++(60:-2) circle(2pt);
\draw [fill=blue] (5,0)++(60:-2)++(0.3,-0.2) circle(0pt) node{$N$};

\draw [fill=blue] (5,0)++(60:-3) circle(2pt);
\draw [fill=blue] (5,0)++(60:-3)++(0.3,-0.2) circle(0pt) node{$M$};

\draw [fill=white] (5,0)++(60:-4) circle(2pt);
\draw [fill=blue] (5,0)++(60:-4)++(0.3,-0.2) circle(0pt) node{$R$};

\draw [fill=blue] (0,0)++(60:3) circle(2pt);
\draw [fill=blue] (0,0)++(60:3)++(0.3,-0.2) circle(0pt) node{$U$};

\draw [fill=blue] (1,0)++(60:3) circle(2pt);
\draw [fill=blue] (1,0)++(60:3)++(0.3,-0.2) circle(0pt) node{$V$};

\draw [fill=blue] (2,0)++(60:3) circle(2pt);
\draw [fill=blue] (2,0)++(60:3)++(0.3,-0.2) circle(0pt) node{$W$};

\draw [fill=white] (3,0)++(60:3) circle(2pt);
\draw [fill=blue] (3,0)++(60:3)++(0.3,-0.2) circle(0pt) node{$X_1$};

\draw [fill=white] (4,0)++(60:3) circle(2pt);
\draw [fill=blue] (4,0)++(60:3)++(0.3,-0.2) circle(0pt) node{$X_2$};

\draw [fill=blue] (2,0)++(60:1) circle(2pt);
\draw [fill=blue] (2,0)++(60:1)++(0.3,-0.2) circle(0pt) node{$S_1$};

\draw [fill=blue] (3,0)++(60:1) circle(2pt);
\draw [fill=blue] (3,0)++(60:1)++(0.3,-0.2) circle(0pt) node{$S_2$};

\draw [fill=white] (4,0)++(60:1) circle(2pt);
\draw [fill=blue] (4,0)++(60:1)++(0.3,-0.2) circle(0pt) node{$S_3$};

\draw [fill=white] (5,0)++(60:1) circle(2pt);
\draw [fill=blue] (5,0)++(60:1)++(0.3,-0.2) circle(0pt) node{$S_4$};

\draw [fill=white] (4,0)++(60:2) circle(2pt);
\draw [fill=blue] (4,0)++(60:2)++(0.3,-0.2) circle(0pt) node{$X$};

\draw [fill=white] (6,0)++(60:-2) circle(2pt);
\draw [fill=blue] (6,0)++(60:-2)++(0.3,-0.2) circle(0pt) node{$Y$};

\end{tikzpicture}
\end{center}
\caption{}\label{f:tcol1}
\end{figure}

\begin{lemma}\label{l:col2}
Let $\mathbb{E}^2$ be coloured in red and blue so that there is no red $\ell_2$. Let $\mathfrak{L}$ be a unit triangular lattice that does not contain three red points forming a $\mathfrak{T}_3$. If there is no blue $\ell_5$, then the colouring of $\mathfrak{L}$ is unique (up to translation or rotation by a multiple of $60\degree$), and is depicted in Figure~\ref{f:lcol}.
\end{lemma}

\begin{figure}[H]
\begin{center}
\begin{tikzpicture}[line cap=round,line join=round,>=triangle 45,x=1.0cm,y=1.0cm,scale=0.7]
\clip(-2.9,0.3) rectangle (12.7,6.5);

\foreach \a in {0,1,...,10}
\draw [dashed, opacity=0.5] (-10,0)++(60:\a)++(120:\a)--+(32,0);

\foreach \a in {0,1,...,10}
\draw [dashed, opacity=0.5] (-10,0)++(60:\a+1)++(120:\a)--+(32,0);

\foreach \a in {0,1,...,30}
\draw [dashed, opacity=0.5] (-10,0)++(\a,0)--+(60:10);

\foreach \a in {0,1,...,30}
\draw [dashed, opacity=0.5] (-10,0)++(\a,0)--+(120:10);

\foreach \b in {0,1,...,20}
\foreach \a in {0,1,...,20}
\draw [fill=blue] (-5,0)++(\a,0.)++(60:\b)++(120:\b) circle (2pt);

\foreach \b in {0,1,...,20}
\foreach \a in {0,1,...,20}
\draw [fill=blue] (-5,0)++(\a,0)++(60:1)++(60:\b)++(120:\b) circle (2pt);

\foreach \a in {0,...,5}
\draw [fill=red] (0,0)++(60:\a)++(120:\a) node[diamond, scale=0.7, fill=red] {};

\foreach \a in {0,...,4}
\draw [fill=red] (2,0)++(60:\a+1)++(120:\a) node[diamond, scale=0.7, fill=red] {};

\foreach \a in {0,...,5}
\draw [fill=red] (5,0)++(60:\a)++(120:\a) node[diamond, scale=0.7, fill=red] {};

\foreach \a in {0,...,4}
\draw [fill=red] (7,0)++(60:\a+1)++(120:\a) node[diamond, scale=0.7, fill=red] {};

\foreach \a in {0,...,4}
\draw [fill=red] (12,0)++(60:\a+1)++(120:\a) node[diamond, scale=0.7, fill=red] {};

\foreach \a in {0,...,4}
\draw [fill=red] (-3,0)++(60:\a+1)++(120:\a) node[diamond, scale=0.7, fill=red] {};

\foreach \a in {0,...,5}
\draw [fill=red] (10,0)++(60:\a)++(120:\a) node[diamond, scale=0.7, fill=red] {};

\end{tikzpicture}
\end{center}
\caption{}\label{f:lcol}
\end{figure}

\begin{proof}
Suppose that $\mathbb{E}^2$ is coloured in red and blue so that there is no red $\ell_2$ and no blue $\ell_5$.
If $\mathfrak{L}$ does not contain a red point, then any $\ell_5$ is blue, therefore $\mathfrak{L}$ contains a red point $A$. By Lemma~\ref{l:bluetr}, one of the points of $\mathfrak{L}$ at distance $\sqrt{3}$ to $A$ is red (otherwise the three such points form a blue triangle with side length $3$ and red centre $A$). Denote this point by $B$ (Figure~\ref{f:lcol1}). Since 
$\mathfrak{L}$ does not contain a red $\mathfrak{T}_3$, the points $D$ and $G$ are blue. Points $E$, $F$, $I$, $H$, $K$, $J$ are blue, since they are distance $1$ apart from $B$. Then the point $B'$ is red (otherwise blue $\ell_5$ $DEFGB'$ is formed). Point $N$ is $1$ apart from $B'$, hence blue. Then $C$ and $A'$ are red (otherwise a blue $\ell_5$ is formed).\\
By repeating the same argument for points $B$ and $C$, $B$ and $A$ (instead of $A$ and $B$), and so on, it can be shown that any node of $\mathfrak{L}$ on the line $AB$ is red. Similarly, since $A'$ and $B'$ are both red, any node of $\mathfrak{L}$ on the line $A'B'$ is red. By the same argument, $A''$, $B''$ and any node on the line containing them is red; $A'''$, $B'''$ and any node on the line containing them is red, and so on. By colouring all point distance $1$ apart form red points blue, the colouring in Figure~\ref{f:lcol} is obtained.
\end{proof}

\begin{figure}[H]
\begin{center}
\begin{tikzpicture}[line cap=round,line join=round,>=triangle 45,x=1.0cm,y=1.0cm,scale=1]
\clip(-5.5,-1.5) rectangle (3,4);

\foreach \a in {1,2,...,10}
\draw [dashed, opacity=0.3] (-10,0)++(240:10)++(60:\a)++(120:\a)--+(32,0);

\foreach \a in {1,2,...,10}
\draw [dashed, opacity=0.3] (-10,0)++(240:10)++(60:\a+1)++(120:\a)--+(32,0);

\foreach \a in {1,2,...,30}
\draw [dashed, opacity=0.3] (-10,0)++(240:10)++(\a,0)--+(60:30);

\foreach \a in {1,2,...,30}
\draw [dashed, opacity=0.3] (-10,0)++(240:10)++(\a,0)--+(120:30);

\draw [fill=red] (0,0) node[diamond, scale=0.7, fill=red] {};
\draw [fill=blue] (0,0)++(0.3,-0.2) circle(0pt) node{$A$};

\draw [fill=red] (0,0)++(60:1)++(120:1) node[diamond, scale=0.7, fill=red] {};
\draw [fill=blue] (0,0)++(60:1)++(120:1)++(0.3,-0.2) circle(0pt) node{$B$};

\draw [fill=white] (-5,0) circle(2pt);
\draw [fill=blue] (-5,0)++(0.3,-0.2) circle(0pt) node{$A'''$};

\draw [fill=white] (-5,0)++(60:1)++(120:1) circle(2pt);
\draw [fill=blue] (-5,0)++(60:1)++(120:1)++(0.3,-0.2) circle(0pt) node{$B'''$};

\draw [fill=white] (2,0)++(120:4) circle(2pt);
\draw [fill=blue] (2,0)++(120:4)++(0.3,-0.2) circle(0pt) node{$C$};

\draw [fill=blue] (2,0)++(120:3) circle(2pt);
\draw [fill=blue] (2,0)++(120:3)++(0.3,-0.2) circle(0pt) node{$H$};

\draw [fill=blue] (2,0)++(120:2) circle(2pt);
\draw [fill=blue] (2,0)++(120:2)++(0.3,-0.2) circle(0pt) node{$I$};

\draw [fill=blue] (2,0)++(120:1) circle(2pt);
\draw [fill=blue] (2,0)++(120:1)++(0.3,-0.2) circle(0pt) node{$G$};

\draw [fill=white] (2,0)++(120:0) circle(2pt);
\draw [fill=blue] (2,0)++(120:0)++(0.3,-0.2) circle(0pt) node{$N$};

\draw [fill=white] (2,0)++(120:-1) circle(2pt);
\draw [fill=blue] (2,0)++(120:-1)++(0.3,-0.2) circle(0pt) node{$A'$};

\draw [fill=blue] (1,0)++(120:1) circle(2pt);
\draw [fill=blue] (1,0)++(120:1)++(0.3,-0.2) circle(0pt) node{$F$};

\draw [fill=blue] (0,0)++(120:1) circle(2pt);
\draw [fill=blue] (0,0)++(120:1)++(0.3,-0.2) circle(0pt) node{$E$};

\draw [fill=blue] (-1,0)++(120:1) circle(2pt);
\draw [fill=blue] (-1,0)++(120:1)++(0.3,-0.2) circle(0pt) node{$D$};

\draw [fill=white] (-2,0)++(120:1) circle(2pt);
\draw [fill=blue] (-2,0)++(120:1)++(0.3,-0.2) circle(0pt) node{$B''$};

\draw [fill=white] (3,0)++(120:1) circle(2pt);
\draw [fill=blue] (3,0)++(120:1)++(0.3,-0.2) circle(0pt) node{$B'$};

\draw [fill=blue] (-2,0)++(60:2) circle(2pt);
\draw [fill=blue] (-2,0)++(60:2)++(0.3,-0.2) circle(0pt) node{$J$};

\draw [fill=blue] (-2,0)++(60:3) circle(2pt);
\draw [fill=blue] (-2,0)++(60:3)++(0.3,-0.2) circle(0pt) node{$K$};

\draw [fill=white] (-2,0) circle(2pt);

\draw [fill=white] (-2,0)++(60:-1) circle(2pt);
\draw [fill=blue] (-2,0)++(60:-1)++(0.3,-0.2) circle(0pt) node{$A''$};

\end{tikzpicture}
\end{center}
\caption{}\label{f:lcol1}
\end{figure}

\begin{proof}[Proof of Theorem~\ref{main}]
Let the Euclidean space $\mathbb{E}^2$ be coloured in red and blue so that there are no two red points distance $1$ apart. Suppose that there are no five blue points that form an $\ell_5$. Then there is a red point $A$. Consider two points $B$ and $C$, both distance $5$ apart from $A$, such that $|BC|=1$. At least one of the points $B$ and $C$ (say, $B$) is blue. Consider the unit triangular lattice $\mathfrak{L}$ that contains $A$ and $B$. By Lemma~\ref{l:col1} and Lemma~\ref{l:col2}, $\mathfrak{L}$ is coloured either as in Figure~\ref{f:tcol} or as in Figure~\ref{f:lcol}. But neither one of the colourings contains two points of different colour distance $5$ apart, which gives a contradiction. Therefore, there exist five blue points that form an $\ell_5$.
\end{proof}

\section{Acknowledgements}
The author would like to thank Ron Graham and Roz\'alia Juh\'asz for providing information about the current state of the problem, and David Gunderson for valuable comments.


\begin{bibdiv}
\begin{biblist}[\normalsize]

\bib{at}{article}{
    author={Arman, Andrii},
    author={Tsaturian, Sergei},
    title={A result in asymmetric Euclidean Ramsey theory},
    date={2017},
    note={https://arxiv.org/pdf/1702.04799.pdf, accessed 27 March 2017},
}

\bib{erdos1}{article}{
   author={Erd\H os, P.},
   author={Graham, R. L.},
   author={Montgomery, Paul},
   author={Rothschild, B. L.},
   author={Spencer, J.},
   author={Straus, E. G.},
   title={\emph{Euclidean Ramsey theorems. I}},
   journal={\emph{J. Combin. Theory Ser. A}},
   volume={14},
   date={1973},
   pages={341--363},
}

\bib{erdos2}{article}{
   author={Erd\H os, Paul},
   author={Graham, R. L.},
   author={Montgomery, P.},
   author={Rothschild, B. L.},
   author={Spencer, J.},
   author={Straus, E. G.},
   title={\emph{Euclidean Ramsey theorems. II}},
   conference={
      title={\emph{Infinite and finite sets (Colloq., Keszthely, 1973; dedicated
      to P. Erd\H os on his 60th birthday), Vol. I}},
   },
   book={
      publisher={North-Holland, Amsterdam},
   },
   date={1975},
   pages={529--557. Colloq. Math. Soc. J\'anos Bolyai, Vol. 10},
}

\bib{erdos3}{article}{
   author={Erd\H os, P.},
   author={Graham, R. L.},
   author={Montgomery, P.},
   author={Rothschild, B. L.},
   author={Spencer, J.},
   author={Straus, E. G.},
   title={\emph{Euclidean Ramsey theorems. III}},
   conference={
      title={\emph{Infinite and finite sets (Colloq., Keszthely, 1973; dedicated
      to P. Erd\H os on his 60th birthday), Vol. I}},
   },
   book={
      publisher={North-Holland, Amsterdam},
   },
   date={1975},
   pages={559--583. Colloq. Math. Soc. J\'anos Bolyai, Vol. 10},
}

\bib{graham}{book}{
   title={\emph{Euclidean Ramsey theory, in} Handbook of discrete and computational geometry \emph{(J. E. Goodman and J. O'Rourke, Eds.)}},
   edition={2},
   author={R. L. Graham},
   editor={},
   publisher={Chapman \& Hall/CRC, Boca Raton, FL},
   date={2004},
}

\bib{ivan}{thesis}{
   author={Iv\'an, L\'aszl\'o},
   title={Monochromatic point sets in the plane and in the space},
   date={1979},
   note={Master’s Thesis, University of Szeged, Bolyai Institute (in Hungarian)}
}


\bib{juhasz}{article}{
   author={Juh\'asz, Roz\'alia},
   title={\emph{Ramsey type theorems in the plane}},
   journal={\emph{J. Combin. Theory Ser. A}},
   volume={27},
   date={1979},
   pages={152--160},
}

\end{biblist}
\end{bibdiv}
\end{document}